\documentclass[12pt,reqno,a4paper]{amsart}
\usepackage{amsfonts}
\usepackage{epsfig}
\usepackage{graphicx}
\usepackage{amsmath}
\usepackage{amssymb}

\newcommand{\Rr}{\mathbb{R}}
\newcommand{\Nn}{\mathbb{N}}

\newcommand{\A}{\mathcal{A}}
\newcommand{\EE}{\mathcal{E}}
\newcommand{\MM}{\mathcal{M}}
\newcommand{\NN}{\mathcal{N}}
\newcommand{\UU}{\mathcal{U}}

\newcommand{\Per}{\mbox{Per}}

\newcounter{main}

\numberwithin{equation}{section}

\newtheorem{theorem}{Theorem}[section]
\newtheorem{proposition}[theorem]{Proposition}
\newtheorem{lemma}[theorem]{Lemma}

\newtheorem{maintheorem}{Theorem}

\newcommand{\blanksquare}{\,\,\,$\sqcup\!\!\!\!\sqcap$}

\newcounter{example}
{{\stepcounter{example}}{\flushleft {\bf Example \arabic{example}:}}}%
{\par}

\title[Elliptic dynamics of 4-dim Hamiltonians]
{Hamiltonian elliptic dynamics on symplectic $4$-manifolds}

\author[M. Bessa]{M\'{a}rio Bessa}
\address{Centro de Matem\'atica da Universidade do Porto, 
Rua do Campo Alegre, 687, 
4169-007 Porto, Portugal}
\email{bessa@fc.up.pt}

\author[J. Lopes Dias]{Jo\~{a}o Lopes Dias}
\address{Departamento de Matem\'atica, ISEG, 
Universidade T\'ecnica de Lisboa,
Rua do Quelhas 6, 1200-781 Lisboa, Portugal}
\email{jldias@iseg.utl.pt}

\begin{document}

\begin{abstract}
We consider $C^2$ Hamiltonian functions on compact $4$-dimensional symplectic manifolds to study elliptic dynamics of the Hamiltonian flow, namely the so-called Newhouse dichotomy.
We show that for any open set $U$ intersecting a far from Anosov regular energy surface, there is a nearby Hamiltonian having an elliptic closed orbit through $U$.
Moreover, this implies that for far from Anosov regular energy surfaces of a $C^2$-generic Hamiltonian the elliptic closed orbits are generic.
\end{abstract}

\maketitle


\section{Introduction}

Hamiltonian systems form a fundamental subclass of dynamical systems. Their importance follows from the vast range of applications throughout different branches of science.
Generic properties of such systems are thus of great interest since they give us the ``typical'' behaviour (in some appropriate sense) that one could expect from the class of models at hand (cf.~\cite{R}).
There are, of course, considerable limitations to the amount of information one can extract from a specific system by looking at generic cases.
Nevertheless, it is of great utility to learn that a selected model can be slightly perturbed in order to obtain dynamics we understand in a reasonable way.

The well-known Newhouse dichotomy is among such generic properties involving certain classes of conservative dynamical systems.
In a broad sense, it deals with systems whose dynamics can only be either of uniformly hyperbolic type or else the elliptic orbits are dense. 
All other possibilities are excluded.

The Newhouse dichotomy was first proved for $C^1$-generic symplectomorphisms in~\cite{N}, and extensions have appeared afterwards~\cite{A,SX}.
Those were all done for discrete-time dynamics.
More recently, a proof for divergence-free $3$-flows is contained in the work of Bessa and Duarte~\cite{BD} and it is, as far as we are aware, the only complete result of this nature concerning continuous-time systems.
As a matter of fact, the Hamiltonian case is already mentioned in~\cite[section 6]{N}. However, it is only under the form of a statement and a plan based in the symplectomorphisms approach. 
Briefly, that consists in the following three steps: the unfolding of homoclinic tangencies by a weakened version of Takens' homoclinic point theorem;
the consequent implication that for a $C^1$-generic symplectomorphism the homoclinic points are dense; and an argument deriving that, unless Anosov, those maps can be $C^1$-approximated by others yielding an elliptic periodic point in any pre-assigned open set.

Our proof of the Newhouse dichotomy for Hamiltonian flows in symplectic compact $4$-manifolds, is considerably distinct.
Moreover, we actually improve Newhouse's statement~\cite[Theorem 6.2]{N} to get a residual set of energy surfaces holding the dichotomy, as well as filling a gap in the literature concerning this problem.
Our method explores a weak form of hyperbolicity, due to Ma\~n\'e, called \emph{dominated splitting} (see~\cite{BDV} and references therein).
When there is absence of domination, we follow an adaptation of Ma\~n\'e's perturbation techniques contained in~\cite{BD} to create closed elliptic orbits.
On the other hand, away from the uniformly hyperbolic energy levels we can exclude domination. For that we rely on recently available results, viz.
Vivier's Hamiltonian version of Franks' lemma~\cite{V}
and the authors' theorem on the Bochi-Ma\~n\'e's dichotomy for Hamiltonians~\cite{BJD} (see also~\cite{Bo,B2}).



\medskip

Let $(M,\omega)$ be a smooth symplectic compact $4$-manifold with a smooth boundary $\partial M$ (including the case $\partial M=\emptyset$).
We denote $C^s(M,\Rr)$, $2\leq s\leq\infty$, as the set of real-valued functions on $M$ that are constant on $\partial M$, the $C^s$ Hamiltonians. This set is endowed with the $C^2$-topology.

An energy surface is a connected component of an (invariant) level set $H^{-1}(e)$ of a Hamiltonian $H$, where $e\in H(M)\subset\Rr$ is called the energy. It is regular if it does not contain critical points. 
A regular energy surface is Anosov if it is uniformly hyperbolic. 
It is far from Anosov if not in the closure of Anosov surfaces. 
Notice that by structural stability the union of Anosov energies is open.
Moreover, Anosov surfaces do not contain elliptic closed orbits.

\begin{maintheorem}\label{thm 1}
Given an open subset $U\subset M$, if a $C^s$ Hamiltonian, $2\leq s\leq\infty$, has a far from Anosov regular energy surface intersecting $U$, then it can be $C^{2}$-approximated by a $C^\infty$ Hamiltonian having an closed elliptic orbit through $U$.
\end{maintheorem}

As an almost direct consequence we arrive at the Newhouse dichotomy for $4$-dimensional Hamiltonians. Recall that for a $C^2$-generic Hamiltonian all but finitely many points are regular.

\begin{maintheorem}\label{thm 3}
For a $C^2$-generic Hamiltonian $H\in C^s(M,\Rr)$, $2\leq s\leq\infty$, the union of the Anosov regular energy surfaces and the closed elliptic orbits, forms a residual subset of $M$.
\end{maintheorem}

At this stage it is relevant to discuss the recent result (in the same $4$-dimensional setting) that motivated the proof of the Hamiltonian version of Franks' lemma.
Th\'er\`ese Vivier showed in~\cite{V} that any robustly transitive regular energy surface of a $C^2$-Hamiltonian is Anosov. 
Recall that a regular energy surface is \emph{transitive} if it has a dense orbit, and it is \emph{robustly transitive} if the restriction of any sufficiently $C^2$-close Hamiltonian to a nearby regular energy surface is still transitive.
It is easy to see that our results also imply this theorem.
In fact, if a regular energy surface $\mathcal{E}$ of a $C^2$-Hamiltonian $H$ is far from Anosov, then by Theorem~\ref{thm 3} there exists a $C^2$-close $C^\infty$-Hamiltonian with an elliptic closed orbit on a nearby regular energy surface. 
This invalidates the chance of robust transitivity for $H$ according to a KAM-type criterium (see~\cite[Corollary 9]{V}).


\medskip

In section~\ref{section:prelim} we present basic notions and results to be used later for the proofs of the theorems. Those appear in section~\ref{section:proofthm1} for Theorem~\ref{thm 1} and in section~\ref{section:proofthm3} for Theorem~\ref{thm 3}.

\section{Preliminaries}\label{section:prelim}

\subsection{Notations and basic definitions}

We refer to \cite{BJD} for the basic notions in Hamiltonian dynamical systems and the notations used in this paper.
In particular, we recall that given a differentiable Hamiltonian $H\colon M\to\Rr$ we denote $X_H$ as the Hamiltonian vector field, $\varphi_H^t$ the corresponding Hamiltonian flow, and $\Phi_H^t$ the transversal linear Poincar\'e flow.
Notice that $\Phi_H^t$ is an automorphism of a two dimensional vector bundle $\cup_{x\in M}\mathcal{N}_x$. 

For a $\varphi^{t}_{H}$-invariant set $\Lambda\subset{M}$ and $m\in\Nn$ we say that a splitting of the bundle $\mathcal{N}_\Lambda=\mathcal{N}^{-}_\Lambda\oplus\mathcal{N}^{+}_\Lambda$ is an \emph{$m$-dominated splitting} for the transversal linear Poincar\'{e} flow if it is $\Phi^{t}_{H}$-invariant and continuous such that
\begin{equation}\label{dd}
\frac{\|\Phi^{m}_{H}(x)|\mathcal{N}^{-}_{x}\|}{\|\Phi^{m}_{H}(x)|\mathcal{N}^{+}_{x}\|}\leq{\frac{1}{2}},
\qquad
x\in\Lambda.
\end{equation}
We call $\mathcal{N}_\Lambda=\mathcal{N}^{-}_\Lambda\oplus\mathcal{N}^{+}_\Lambda$ a \emph{dominated splitting} if it is $m$-dominated for some $m\in\Nn$.


A metric on the manifold $M$ can be derived in the usual way through the Darboux charts. It will be denoted by $\operatorname{dist}$. Hence we define the open balls $B(p,r)$ of the points $x$ verifying $\operatorname{dist}(x,p)<r$.

\subsection{Elliptic, parabolic and hyperbolic closed orbits}

Let $\Gamma\subset M$ be a closed orbit of least period $\tau$. The characteristic multipliers of $\Gamma$ are the eigenvalues of $\Phi_H^{\tau}(p)$, which are independent of the point $p\in\Gamma$. 
We say that $\Gamma$ is
\begin{itemize}
\item
{\it elliptic} iff the two characteristic multipliers are simple, non-real and of modulus $1$;
\item
{\it parabolic} iff the characteristic multipliers are real and of modulus $1$;
\item
{\it hyperbolic} iff the characteristic multipliers have modulus different from $1$. 
\end{itemize}

It is clear that under small perturbations, elliptic and hyperbolic orbits are stable whilst parabolic ones are unstable.

We refer to a point in a closed orbit as periodic. Periodic points are classified in the same way as the respective closed orbit.

\subsection{Perturbation lemmas}\label{section:PL}

We include here two perturbation results available in the literature that will be used in the remaining sections.
The first is a version of Pugh's closing lemma, stating that the orbit of a non-wandering point can be approximated for a very long time by a closed orbit of a nearby Hamiltonian.

\begin{theorem}[Closing lemma for Hamiltonians~\cite{PR}]\label{closing lemma}
Let $H\in C^s(M,\Rr)$, $2\leq s\leq\infty$, a non-wandering point $x\in M$ and $\epsilon,r,\tau>0$.
Then, we can find $\widetilde H\in C^2(M,\Rr)$, a closed orbit $\Gamma$ of $\widetilde H$ with least period $\ell$, $p\in\Gamma$ and a map $g\colon[0,\tau]\to[0,\ell]$ close to the identity such that:  
\begin{itemize}
\item 
$\|\widetilde H-H\|_{C^2}<\epsilon$,
\item 
$\operatorname{dist}\left( \varphi_{H}^t(x), \varphi_{\widetilde H}^{g(t)}(p)\right) <r $, $0\le t\le \tau$, and
\item 
$H=\widetilde H$ on $M\setminus A$, where $A=\bigcup_{0\le t\le \ell}\big(B(p,r)\cap B(\varphi_{\widetilde H}^t(p),r)\big)$.
\end{itemize}
\end{theorem}

The next theorem is a version of Franks' lemma. Roughly, it says that we can realize a Hamiltonian corresponding to a given perturbation of the transversal linear Poincar\'e flow. It is proved for $2d$-dimensional manifolds with $d\geq2$.

\begin{theorem}[Franks' lemma for Hamiltonians~\cite{V}]\label{Vivier}
Let $H\in C^s(M,\Rr)$, $2\leq s\leq\infty$, $\epsilon,\tau>0$ and $x\in M$.
There exists $\delta>0$ such that for any flowbox $V$ of an injective arc of orbit $\Sigma=\varphi_{H}^{[0,t]}(x)$, $t\geq \tau$, and a $\delta$-perturbation $F$ of $\Phi_{H}^t(x)$, there is $\widetilde H\in C^{\max\{2,s-1\}}(M,\Rr)$ satisfying:
\begin{itemize}
\item
$\|\widetilde H-H\|_{C^2}<\epsilon$,
\item
$\Phi_{\widetilde H}^t(x)=F$,
\item
$H=\widetilde H$ on $\Sigma\cup (M\setminus V)$.
\end{itemize}
\end{theorem}

\section{Proof of Theorem \ref{thm 1}}\label{section:proofthm1}

By Robinson's version of the Kupka-Smale theorem~\cite{R}, a $C^2$-generic Hamiltonian has all closed orbits of hyperbolic or elliptic type. Moreover, it follows from the closing lemma that for a $C^2$-generic Hamiltonian $H$, the set of closed orbits $\Per(H)$ is dense in the nonwandering set of $H$. 
Hence, by Poincar\'e recurrence, $\Per(H)$ is dense in $M$.


\subsection{Hyperbolic orbits with domination}

Let us recall an elementary consequence of the
persistence of dominated splittings; see~\cite{BDV} for the full details.

\begin{lemma}\label{PDD}
Consider $H\in C^s(M,\Rr)$, $2\leq s\leq\infty$, and a $\varphi_H^t$-invariant $\Lambda\subset M$ with $m$-dominated splitting for $\Phi^t_H$.
Then, there exists $\delta>0$ and a neighborhood $V$ of $\Lambda$ such that the set
$$
\bigcap_{t\in\mathbb{R}}\varphi_{H}^t(V)
$$
has an $m'$-dominated splitting for any Hamiltonian $\widetilde H$ $\delta$-$C^2$-close to $H$ for $m'>m$.
\end{lemma}

\subsection{Absence of domination}

The following propositions allow us to use the lack of hyperbolic behaviour of the transversal linear Poincar\'e flow to produce elliptic closed orbits by small perturbations.

\begin{proposition}\label{propA}
Let $H\in C^s(M,\Rr)$, $2\leq s\leq\infty$, and  $\epsilon>0$.
There is $\theta>0$ such that for any closed hyperbolic orbit $\Gamma$ with least period $\tau>1$ and $\measuredangle(\mathcal{N}_{q}^{u},\mathcal{N}_{q}^{s})<\theta$, $q\in\Gamma$, there is $\widetilde H\in C^\infty(M,\Rr)$ $\epsilon$-$C^2$-close to $H$ for which $\Gamma$ is elliptic with least period $\tau$.
\end{proposition}

\begin{proof}
We start by choosing a $C^\infty$ Hamiltonian $H_1$ $C^2$-close to $H$ having a closed orbit nearby $\Gamma$. If $\Gamma$ is elliptic, the proposition is proved. So we assume it is hyperbolic.

In a thin flowbox about $\Gamma$ use local symplectic flowbox coordinates (see~\cite[Theorem 4.1]{BJD}) and the perturbation described in~\cite[Lemma 4.2]{BJD} to interchange $\NN_q^u$ and $\NN_q^s$.
Indeed, we find $\alpha_0$ so that for any $0<\alpha<\alpha_0$, 
$$
R_\theta\NN_q^u=\NN_{\varphi_{H_1}^1(q)}^s,
$$
where $R_\alpha\in \operatorname{SO}(2,\Rr)$ rotates by $\alpha$.
By~\cite[Lemma 3.5]{BD} $\Phi_{H_1}^t(q)R_\theta$ is elliptic.
Then, theorem~\ref{Vivier} guarantees the existence of $\widetilde H$ as claimed.
\end{proof}

\begin{proposition}\label{propB}
Let  $H\in C^s(M,\Rr)$, $2\leq s\leq\infty$, and $\epsilon,\theta>0$. 
There exist $m,T\in\mathbb{N}$ ($T\gg m$) such that, if a hyperbolic closed orbit $\Gamma$ with least period $\tau>T$ satisfies:
\begin{enumerate}
\item 
$\measuredangle(\mathcal{N}_{q}^{u},\mathcal{N}_{q}^{s})\geq \theta$ for all $q\in\Gamma$, and 
\item 
$\Gamma$ has no $m$-dominated splitting,
\end{enumerate}
then we can find $\widetilde H\in C^\infty(M,\Rr)$  $\epsilon$-$C^2$-close to $H$ for which $\Gamma$ is elliptic with least period $\tau$.
\end{proposition}

\begin{proof}
We choose a $C^\infty$ Hamiltonian $H_0$ $C^2$-close to $H$ satisfying the same hypothesis. 
The idea of the proof is to adapt the proofs of the lemmas 3.10-3.12 in~\cite{BD} to the Hamiltonian setting. This requires the use of the perturbation from~\cite[Lemma 4.2]{BJD} considered in local symplectic flowbox coordinates about $\Gamma$ (see~\cite[Theorem 4.1]{BJD}).

So,~\cite[Lemma 3.10]{BD} implies that there exists $m(\epsilon,\theta)\in\mathbb{N}$, such that given any periodic orbit $q$ of period $\tau(q)>m$ satisfying (1) and (2) above, we can find a Hamiltonian $H_{1}$ $\epsilon/3$-$C^2$-close such that
\begin{equation}\label{prop ii}
\Phi_{H_{1}}^{m}(\mathcal{N}_{q}^{u})=\mathcal{N}_{\varphi_{H}^{m}(q)}^{s}.
\end{equation}

By~\cite[Lemma 3.11]{BD}, there exists $K=K(\theta,m)>0$ and $H_{2}$ $\epsilon/3$-$C^2$-close to $H_{1}$ such that
\begin{equation}\label{prop II}
\|\Phi_{H_{2}}^{\tau}(q)\|<K,
\end{equation}
where $q\in\Gamma$ and $\Gamma$ is any closed orbit of period $\tau$ satisfying the statements (1) and (2) of the lemma.
Notice that \eqref{prop ii} let us blend different expansion rates, thus obtaining $K$ not depending on how large is $\tau$.

Now, define $T>m$ depending on $K$ as in ~\cite[Lemma 3.12]{BD}, and take $\tau>T$. Let $q\in M$ be a periodic point of period $\geq\tau$ and, for simplicity, assume that $\tau\in\mathbb{N}$.
We concatenate its orbit $\Gamma$ in $\tau$ time-one disjoint intervals. In each one of these intervals we consider an abstract symplectic action shrinking by a factor of order $\epsilon$ along the direction $\mathcal{N}^{u}$ and stretching by a factor of order $\epsilon$ along the direction $\mathcal{N}^{s}$. 
Then, we use Theorem~\ref{Vivier} to realize this perturbation by a Hamiltonian $\epsilon/3$-$C^2$-close. 
Since this procedure is repeated $\tau$-times and $\tau$ is very large while $K$ remains bounded (see \eqref{prop II}), it follows that there exists a perturbation $\widetilde H$ such that $\|\Phi_{\widetilde H}^{\tau}(q)\|=1$ and $q$ is of elliptic type.

\end{proof}

\subsection{Proof of Theorem~\ref{thm 1}}

Take $H\in C^s(M,\Rr)$, $2\leq s\leq\infty$, $\epsilon>0$, a regular energy surface $\EE$ far from Anosov and an open subset $U\subset M$ such that $\EE\cap U\not=\emptyset$.

As mentioned earlier, for a $C^2$-generic Hamiltonian closed orbits are dense and there are no parabolic ones.
We further assume that there are no closed elliptic orbits through $U$. 

For a hyperbolic closed orbit $\Gamma$, consider the angle between the stable and unstable directions.
If it is smaller than $\theta$ given by Proposition~\ref{propA}, then we have a Hamiltonian $\epsilon$-$C^2$-close for which $\Gamma$ is elliptic.
Now, let $m$ and $T$ as in Proposition~\ref{propB}.
If $\Gamma$ has angle larger than $\theta$ with least period greater than $T$ and no $m$-dominated splitting, then $\Gamma$ is again elliptic for a $\epsilon$-$C^2$-close Hamiltonian.

Therefore, it remains to consider the case of all closed orbits through ${U}$ being hyperbolic with angle greater than $\theta$ and with $m$-dominated splitting. We will show that this is unattainable for a $C^2$-generic Hamiltonian.

By Lemma~\ref{PDD}, the above parameters $m$ and $\theta$ vary continuously with $H$. In fact, for $\widehat{\epsilon}>0$ small enough, every closed orbit associated to a $\widehat{\epsilon}$-$C^2$-close Hamiltonian has uniform parameters of dominated splitting, still denoted by $m$ and $\theta$.

 From~\cite[Theorem 1]{BJD}, and since $\mathcal{E}$ is not approximated by an Anosov regular energy surface, there exists a Hamiltonian $H_{2}$ $\widehat{\epsilon}/3$-$C^2$-close to $H$ and a regular energy surface $\widehat{\EE}$ (arbitrarly close to $\EE$) of $H_2$ having zero Lyapunov exponents $\mu_{\widehat{\EE}}$-a.e. for $\varphi_{H_2}^t$. Here $\mu_{\widehat\EE}$ is the natural induced invariant measure on the energy level sets (see~\cite{BJD}). Denote by $Z\subset\widehat\EE$ the full measure subset of such points with zero Lyapunov exponents.
Using Oseledets' theorem (cf. e.g.~\cite{BJD}), for $x\in Z$ and any $\delta>0$, there exists $t_{x}\in\mathbb{R}$ such that 
\begin{equation}\label{Oseledets}
e^{-\delta t}<\|\Phi_{H_{2}}^{t}(x)\|<e^{\delta t}
\quad\text{whenever}\quad
t\geq t_{x}.
\end{equation}

Let $P\subset U\cap Z$ be the positive measure (same as the one of $U\cap Z$) Poincar\'e recurrence subset associated to $H_{2}$ and $\mu_{\widehat{\EE}}$.
Then, for $x\in P$, there exists a strictly increasing sequence of integers $n_k(x)$ such that $\varphi_{H_{2}}^{n_k(x)}(x) \in U\cap Z$ and 
$$
\lim_{k\to\infty}\varphi_{H_{2}}^{n_k(x)}(x) = x.
$$
Let $\mathcal{T}$ denote the set of positive return times to $U\cap Z$ by $\varphi_{H_{2}}^t$.

Take $\tau\in \mathcal{T} \cap [t_{x},+\infty)$. By the closing lemma (Theorem~\ref{closing lemma}), the $\varphi_{H_{2}}^t$-orbit of $x$ can be approximated up to the (very long) time $\tau$ by a closed orbit $\Gamma$ of a $\widehat\epsilon/3$-$C^2$-close Hamiltonian $H_{1}$.
If we take a small enough $r>0$ in the closing lemma, then \eqref{Oseledets} gives 
$$
e^{-\delta \tau}<\|\Phi_{H_{1}}^{\tau}(p)\|<e^{\delta \tau}
\quad\text{with}\quad
p\in\Gamma.
$$
Notice that $\tau\approx \ell$, the least period of $\Gamma$.

We have thus created a closed orbit where $\Phi_{H_1}^\tau(p)$ is ``weakly'' hyperbolic, i.e. the exponent $\delta$ is as small as we want.
We now perturb it
to cancel the remaining hyperbolicity, and create by Franks' lemma -- Theorem~\ref{Vivier} an elliptic closed orbit.
It gives a Hamiltonian $H_{0}$ $\widehat{\epsilon}/3$-$C^2$-close verifying those conditions. 
This contradicts the fact asserted earlier that every Hamiltonian $\widehat{\epsilon}$-$C^2$-close to $H$ has closed orbits with uniform parameters of dominated splitting $m$ and $\theta$. 
The proof of Theorem~\ref{thm 1} is therefore complete.

\section{Proof of Theorem \ref{thm 3}}\label{section:proofthm3}

Consider the set 
$$
\MM=M\times C^s(M,\Rr)
$$
endowed with the standard product topology. Recall that it is used the $C^2$ topology on the second component.
Given $p\in M$, $\EE_p(M)$ is the energy surface passing through $p$.
The subset
$$
\A=\left\{(p,H)\in\MM\colon \EE_p(H) \text{ is an Anosov regular energy surface} \right\}.
$$
is open by structural stability.
Let $\overline{\A}$ be its closure with complement $\NN=\MM\setminus \overline{\A}$.


Given $\epsilon>0$ and an open set $\mathcal{U}\subset \mathcal{N}$, define the subset $\mathcal{O}(\mathcal{U},\epsilon)$ of pairs $(p,H)\in\mathcal{U}$ for which $H$ has a closed elliptic orbit through the $3$-dim ball $B(p,\epsilon)\cap \EE_p(H)$. 
It follows from Theorem~\ref{thm 1} and the fact that in the $4$-dim case elliptic orbits are stable, that $\mathcal{O}(\mathcal{U},\epsilon)$ is dense and open in $\mathcal{U}$.

Let $\epsilon_k$ be a positive sequence converging to zero. 
Then, define recursively the sequence of dense and open sets $\mathcal{U}_0 = \mathcal{N}$ and
$$ 
\mathcal{U}_{k}=\mathcal{O}(\mathcal{U}_{k-1}, \epsilon_{k-1}),
\quad k\in\Nn.
$$
Notice that $\cap_{k\in\Nn}\UU_k$ is the set of pairs $(p,H)$ yielding the property that $p$ is a periodic elliptic point for $H$.

Finally, the above implies that $\A\cup\UU_k$ is open and dense in $\MM$, and
$$
\mathfrak{A}:=\bigcap_{k\in\Nn}(\A\cup\UU_k)=\A\cup \bigcap_{k\in\Nn}\UU_k
$$
is residual.
By~\cite[Proposition A.7]{BF}, we write 
$$
\mathfrak{A}=\bigcup_{H\in\mathfrak{R}}\mathfrak{M}_H\times \{H\},
$$
where $\mathfrak{R}$ is $C^2$-residual in $C^s(M,\Rr)$ and, for each $H\in\mathfrak{R}$, $\mathfrak{M}_H$ is a residual subset of $M$, having the following property:
if $H\in\mathfrak{R}$ and $p\in\mathfrak{M}_H$, then $\EE_p(H)$ is Anosov or $p$ is a periodic elliptic point.

\section*{Acknowledgements}

We would like to thank Pedro Duarte for his useful suggestions. 
MB was supported by Funda\c c\~ao para a Ci\^encia e a Tecnologia, SFRH/BPD/
20890/2004. 
JLD was partially supported by Funda\c c\~ao para a Ci\^encia e a Tecnologia through the Program FEDER/POCI~2010. 


\end{document}